\documentclass{article}
\pdfoutput=1
\usepackage[utf8]{inputenc}
\usepackage[newzealand]{babel}
\usepackage{amsmath}
\usepackage{amsthm}
\usepackage[author-year]{amsrefs}
\usepackage{amssymb}
\usepackage{tikz}
\usetikzlibrary{arrows,calc,intersections,shapes}
\usepackage{verbatim}
\usepackage{graphicx}
\usepackage[subrefformat=parens]{subcaption}
\usepackage{booktabs}
\usepackage{aliascnt}
\usepackage{enumitem}
\usepackage[breaklinks=true,pdfborder={0 0 0},pdfdisplaydoctitle=true,
pdfpagelabels=true]{hyperref}
\hypersetup{pdfauthor={Stefan Kranich},pdftitle={Generation of real
algebraic loci via complex detours}}

\usepackage{microtype}

\theoremstyle{plain}
\newtheorem{theorem}{Theorem}[section]

\newcommand{\mynewtheorem}[2]{
    \newaliascnt{#1}{theorem}
    \newtheorem{#1}[#1]{#2}
    \aliascntresetthe{#1}
    \expandafter\def\csname #1autorefname\endcsname{#2}
}

\mynewtheorem{lemma}{Lemma}
\theoremstyle{definition}
\mynewtheorem{assumption}{Assumption}
\mynewtheorem{construction}{Construction}
\mynewtheorem{algorithm}{Algorithm}
\mynewtheorem{conjecture}{Conjecture}
\mynewtheorem{definition}{Definition}
\mynewtheorem{example}{Example}
\mynewtheorem{problem}{Problem}
\theoremstyle{remark}
\mynewtheorem{remark}{Remark}
\makeatletter
\let\c@figure\@undefined
\let\c@table\@undefined
\makeatother
\newaliascnt{figure}{theorem}
\aliascntresetthe{figure}
\numberwithin{figure}{section}
\newaliascnt{table}{theorem}
\aliascntresetthe{table}
\numberwithin{table}{section}

\begin{document}

\title{Generation of real algebraic loci via complex~detours}
\author{Stefan Kranich\footnote{Zentrum Mathematik (M10), Technische
Universität München, 85747~Garching, Germany; E-mail address:
\url{kranich@ma.tum.de}}}

\maketitle

\begin{abstract}
\noindent We discuss the locus generation algorithm used by the dynamic
geometry software Cinderella, and how it uses complex detours to resolve
singularities.
We show that the algorithm is independent of the clockwise or
anticlockwise orientation of its complex detours.
We conjecture that the algorithm terminates if it takes small
enough complex detours and small enough steps on every complex detour.
Moreover, we introduce a variant of the algorithm that possibly
generates entire real connected components of real algebraic loci.
Several examples illustrate its use for organic generation of real
algebraic loci.
Another example shows how we can apply the algorithm to simulate
mechanical linkages. Apparently, the use of complex detours produces
physically reasonable motion of such linkages.
\end{abstract}

\section{Introduction}

A locus is a set of points in the plane with a common geometric
property. For example, a circle is the set of points whose distance from
the centre of the circle equals the radius of the circle.
Here we focus on loci that are closed curves generated by dynamic
geometric constructions.

The generation of loci has a very long history. We roughly follow the
exposition of~\cite{BrieskornKnoerrer1986}*{Chapter~1}. Many ancient Greek
mathematicians studied geometric constructions. Some classical
problems, e.g.\ squaring the circle, duplication of the cube (Delian
problem), and angle trisection, resisted any attempt of solution using
only compass and straightedge.
Unlike the ancient Greeks, we know that these problems cannot be solved
by compass and straightedge constructions.\footnote{
\ocite{Wantzel1837} showed that using compass and straightedge, we can only
construct points with coordinates in a field extension of $\mathbb{Q}$
obtained by adjoining all roots of degree a power of two of rational
numbers. He proved that duplication of the cube and angle trisection
reduce to solving irreducible cubic equations, which is therefore
impossible using only compass and straightedge.

Squaring the circle reduces to constructing $\pi$. In the nineteenth
century, it was known that $\pi$ is irrational. Mathematicians wondered
whether it could be a solution of an equation of degree a power of two
that was potentially solvable using only compass and straightedge by
iterative construction of square roots. \ocite{Lindemann1882} settled
the case, showing that $\pi$ is transcendental, i.e.\ not a root of any
algebraic equation.}

After many failed attempts using only compass and straightedge, the
ancient Greeks sought other means of solving these problems.
Menaechmus (c.~350~BC) found that duplication of the cube could be
performed using compass, straight\-edge and conic sections.
Other plane algebraic curves classically used as devices for duplication
of the cube and angle trisection include the cissoid of Diocles
(c.~180~BC) and the conchoid of Nicomedes (c.~180~BC).
Diocles and Nicomedes solved duplication of the cube and angle trisection by
intersecting these curves with other geometric elements. In order to
find the intersections, it is essential that we can construct not only
some points but whole arcs of these curves in a continuous manner.

To that end, we can use dynamic compass and straightedge constructions. We
move one element of the construction, e.g.\ we rotate a line about a
point or slide a point on a line/circle, and follow a point constructed
from the moving element as it traces a curve. If the construction steps
can be described algebraically, then the resulting curve is a real
connected component of a plane algebraic curve.
Such constructions for various curves were known to the ancient
Greeks. Newton, who investigated this technique, called it `organic
generation'.

We can use dynamic geometry software to carry out such constructions. Some
applications allow us to select a line through a point or a point on a
line/circle, which shall be the moving element (mover), and a dependent
point, which shall be followed as it traces a curve (tracer).
The software then attempts to automatically generate the real connected
component of the real plane algebraic curve that is the locus of the
tracer under movement of the mover.

Depending on the underlying model of geometry (and depending on the
algorithm used), the software may fail to generate the entire real
connected component of the real algebraic curve.
\ocite{Kortenkamp1999}*{Section~6.2, esp.~Theorem~6.8} shows that dynamic
geometry systems with geometric primitives like `intersection of a circle
with a line' or `angle bisector of two lines' cannot be both determined
and continuous: If we require that we can determine a unique instance
of a construction for every possible position of its free elements
(movable elements), then we cannot expect that its dependent elements
always move continuously.

The reason behind this is that `intersection of a circle and a line' or
`angle bisector of two lines' are ambiguous geometric operations. In
general, a circle and a line have two (possibly complex) intersections. In
general, two lines have two angle bisectors, which are perpendicular to
each other. Consider an angle bisector of two (unoriented) lines $a$,
$b$ through a common point $P$.
If we rotate $a$ about $P$ and the angle bisector
moves continuously, then the angle bisector rotates at half the angular
velocity. When $a$ has rotated by an angle of $\pi$, it has reached its
initial position. At the same time, the angle bisector has rotated by an
angle of $\pi/2$ and has become perpendicular to its initial position. We
have moved continuously from one possible instance of the construction
to the other. Therefore the dynamic geometric system cannot be determined.

For that reason, we may sometimes observe dependent elements jump
unexpectedly in most dynamic geometry software.
If a tracer jumps during locus generation, the algorithm may miss part
of its real algebraic locus.

As we have seen, if we want to avoid jumping elements, we need to take all
possible output elements of ambiguous geometric operations into account.
We parameterize the motion of the mover using a time parameter $t$. There
may be some points in time when part of the construction degenerates
or when two possible choices of a dependent element coincide and become
indistinguishable. We call such points in time singularities. For example,
the two intersections of a circle and a line coincide when we move the
centre of the circle such that it merely touches the line.
If we can somehow avoid singularities, then we can always distinguish 
all possible choices. Hence, we may be able to determine the instance
that produces a continuous evolution of the construction.

In order to avoid jumping elements in their dynamic geometry software
Cinderella~\cite{KortenkampRichterGebert2006}, Kortenkamp
and Richter-Gebert introduced the paradigm of `complex
detours' \citelist{\cite{Kortenkamp1999}*{esp. Chapter~7}
\cite{KortenkampRichterGebert2001b}
\cite{KortenkampRichterGebert2002}}:
We do not let the time parameter run along the real time axis where we
often encounter singularities. Instead, we embed the real time axis
as the real axis of the complex plane. Between two points on the real
time axis, we let the time parameter take a detour through the complex
plane to circumvent singularities. Thus we can avoid singularities with
high probability.

The locus generation algorithm of Cinderella exploits this principle;
it has been working very well in practice for many years, but from a
theoretical perspective, the algorithm has not been examined in more
detail yet~\cite{RichterGebert2014}.
In what follows, we analyze the algorithm and some assumptions on which
it is based.
Moreover, we introduce a variant of the algorithm that might always
generate an entire real connected component of a real algebraic locus.

\section{A locus generation algorithm}

The locus generation algorithm addresses the following problem: Consider a
geometric construction. Choose an element of the construction whose
movement is constrained to one dimension, e.g.\ a line through a
point, a point on a line, or a point on a circle. We call this element
`mover'. Choose a point of the construction whose position
depends on the position of the mover. We call this point `tracer'.
Suppose the tracer can be constructed from the mover (and possibly
further elements) using only geometric operations that have an algebraic
representation. Then movement of the mover causes the tracer to move on
a real plane algebraic curve.
The goal of the locus generation algorithm is to automatically produce
the locus of the tracer under movement of the mover.

The following locus generation algorithm has two variants. Variant~A
is essentially equivalent to the locus generation algorithm
implemented in Cinderella. Variant~B possibly generates an
entire real connected component of a real algebraic locus
(see~\autoref{cnj:connected-component}).

\begin{algorithm}[locus generation algorithm]
\label{alg:locus-generation}
We rationally parameterize the motion of the mover using a time parameter $t$.
We assume that we start at a non-singular initial time $t_0 \in
\mathbb{R}$ when the position of the tracer is real-valued.
\begin{enumerate}
\item Let $s := 1$, $t := t_0$.
\item Let $t' := t + s \cdot \varepsilon$ for some small step size
$\varepsilon > 0$.
\end{enumerate}
Consider the circle $c$ in the complex plane that has the segment
between $t$ and $t'$ as its diameter. We let the time parameter take a
complex detour on this circle.
\begin{enumerate}[resume]
\item Let $t$ take a small step on circle $c$ in anticlockwise
direction.
\item Update the position of the mover according to its parameterization
in $t$.
\item Update the rest of the construction. In doing so, for ambiguous
elements, use proximity between old and new positions of the elements
to determine the right instance of the construction.
\item \begin{description}[leftmargin=0pt]
\item[Variant A] If $t$ is real-valued and the tracer has real-valued
coordinates, the tracer has reached a point of the real plane algebraic
locus. It may happen that $t$ ends up in its initial position on circle
$c$. In this case, we invert the direction of movement of the mover. To
that end, we set $s := -s$. (Note that this affects only the choice of
sampling points of time parameter $t$; the anticlockwise orientation of
complex detours remains the same.)
If $t = t_0$, $s = 1$, and the tracer has reached its initial position again,
we stop. Otherwise, we go to step~2.

If $t$ is not real-valued or the tracer does not have real-valued
coordinates, we go to step~3.
\item[Variant B] If the tracer has real-valued coordinates, it has
reached a point of the real algebraic locus. It may happen that $t$ is
not real-valued or that it ends up in its initial position on circle
$c$. In any case, we update the direction of movement of the mover.
To that end, we set \[s = \frac{t - a}{|t - a|},\] where $a$ denotes
the centre of circle $c$.
If $t = t_0$ and the tracer has reached its initial position again,
in the initial direction of movement of the time parameter, we stop.
Otherwise, we go to step~2.

If the tracer does not have real-valued coordinates, we go to step~3.
\end{description}
\end{enumerate}
\end{algorithm}

\begin{remark}
\label{rem:tracing-by-proximity}
\autoref{alg:locus-generation} evaluates the geometric construction only
at discrete points in time along a complex detour.
At every sampling point, the
algorithm should select the right instance of the construction, i.e.\
the instance that yields a continuous evolution of the construction. More
precisely, we want that the coordinates of the elements of the construction
are locally analytic functions of time; the algorithm should select
the instance that results from analytic continuation of the coordinate
functions along the complex detours.

If the coordinates are analytic functions of time, then in particular
they are continuous. Hence, if we take a small step in time along the
complex detour, the coordinates change only little. If the step is
small enough then we can select the right instance (in the above sense)
by proximity, i.e.\ we choose every ambiguous element so that its
coordinates change least.
\end{remark}

\noindent Besides the assumption that we start at a non-singular
initial time when the position of the tracer is real-valued,
\autoref{alg:locus-generation} is based on the following assumptions.
(This list of assumptions may not be complete;
\autoref{alg:locus-generation} may be based on further implicit
assumptions.)

\begin{assumption}
\label{ass:at-most-one-singularity}
We assume that we always choose $\varepsilon$ small enough so that
the complex detours wind around at most one ramification point of a
coordinate of the tracer, and only around ramification points at which
the position of the tracer is real-valued.
Thus we preclude that the tracer jumps from one real arc of the real
algebraic locus to another due to a complex detour around a ramification
point of a coordinate at which the tracer has a complex position close
to the real algebraic locus.
\end{assumption}

\begin{assumption}
\label{ass:tracing-by-proximity}
We assume that the steps we take on the complex detour are small enough
so that we can choose the right instance of the construction by proximity
(see~\autoref{rem:tracing-by-proximity}).
\end{assumption}

\begin{assumption}
\label{ass:no-miss}
We assume that the steps we take on the complex detours are small enough
so that we do not miss a real position of the tracer.
\end{assumption}

\begin{remark}
While we usually satisfy \autoref{ass:tracing-by-proximity} in practice,
it is very difficult to guarantee that it is satisfied in general. We
cannot go into details here but refer
to~\cite{KortenkampRichterGebert2002}.
\end{remark}

\begin{remark}
While we usually satisfy \autoref{ass:at-most-one-singularity} in
practice, it is not clear how to guarantee that it is satisfied in
general.
\end{remark}

\begin{remark}
For Variant~A of \autoref{alg:locus-generation}, we can
satisfy~\autoref{ass:no-miss}, if we ensure that time parameter $t$
always attains the real values on the complex detours.
For Variant~B of \autoref{alg:locus-generation}, we must additionally
determine complex points in time on the complex detours when the coordinates
of the tracer are real-valued. We can try to approximate these points in
time by bisection if the imaginary part of a coordinate of the tracer
changes sign or has small absolute value; this may be really difficult
in practice.
\end{remark}

\begin{remark}
Since a real algebraic locus can contain points at infinity, it may sometimes
be advantageous to describe the plane algebraic curve containing the real
algebraic locus by a homogeneous equation $f(x, y, z) = 0$. Thus we can
express points of the locus at infinity using finite coordinates ${(x,
y, z)}^\top$.
We can homogenize an affine equation $f (x, y) = 0$ of total degree $n$
by plugging in $x = x/z$, $y = y/z$, and multiplying by $z^n$. If we
set $z = 1$, we return to the affine equation.
\end{remark}

\begin{remark}
If we use homogeneous coordinates ${(x, y, z)}^\top$ to describe the
position of the tracer, we require that none of $x$, $y$, and $z$ become
infinite; if necessary, we constantly normalize ${(x, y, z)}^\top$ to
avoid that one of the coordinates grows too much.
\end{remark}

\section{Orientation of complex detours}

In Step~3 of \autoref{alg:locus-generation}, we specify that time
parameter $t$ takes complex detours along circles in the complex plane
in anticlockwise direction.
In principle, $t$ can also take complex detours in clockwise direction.
Are there constructions where the generated locus changes depending on
the clockwise or anticlockwise orientation of the complex detours?

In order to answer this question, we make some assumptions on the
real algebraic loci to which we apply our locus generation algorithm,
without loss of generality. Let \[\mathcal{C}\colon f (x,y) = 0\] denote
a plane algebraic curve containing such a locus.

\begin{assumption}
\label{ass:irreducible}
We assume that $f(x,y)$ is irreducible. An analytic transition from one
irreducible component of $f(x,y)$ to another is impossible. Therefore,
if $f(x,y)$ is reducible, we can without loss of generality consider
the irreducible component containing the starting point.
(Note that the assumption that we start \autoref{alg:locus-generation}
at a non-singular initial time precludes that we start at an intersection
of irreducible components.)
\end{assumption}

\begin{assumption}
\label{ass:real-component}
We assume that $\mathcal{C}$ has a real connected component. Otherwise,
the locus is not a real connected component of a real plane algebraic
curve and the locus generation algorithm is not applicable.
\end{assumption}

\begin{lemma}
\label{lem:real-coefficients}
Under \autoref{ass:irreducible} and \autoref{ass:real-component},
without loss of generality, $f (x,y)$ has only real coefficients.
\end{lemma}

\begin{proof}
Conversely, suppose $f(x,y)$ has complex coefficients. Then we can write
$f(x,y)$ in the form
$f(x,y) = f_\Re(x,y) + \mathrm{i} f_\Im(x,y)$ such that the real part
polynomial $f_\Re(x,y)$ and the imaginary part polynomial $f_\Im(x,y)$
possess only real coefficients.
By the assumption that $f(x,y)$ has complex coefficients, $f_\Im(x,y)$
does not vanish identically. If $f_\Re(x,y)$ vanishes identically, then
$f(x,y) = \mathrm{i} f_\Im(x,y)$ and we can cancel the unit $\mathrm{i}$
from the equation $f(x,y) = 0$ to obtain an equation for $\mathcal{C}$
with only real coefficients. Hence, suppose that both real part polynomial
and imaginary part polynomial do not vanish
identically. By~\autoref{ass:real-component}, $\mathcal{C}$ has a
real connected component, i.e.\ over a real interval of $x$-values, $f(x,y)$
vanishes for real $y$-values. This can only be the case if $f_\Re(x,y)$
and $f_\Im(x,y)$ vanish there, i.e.\ if $f_\Re(x,y)$ and $f_\Im(x,y)$
have infinitely many common zeros.
If $f_\Re(x,y)$ and $f_\Im(x,y)$ have infinitely many common zeros,
then they must have a common component or $f_\Im(x,y)$ must be a
non-zero real multiple of $f_\Re(x,y)$. By \autoref{ass:irreducible},
$f_\Re(x,y)$ and $f_\Im(x,y)$ are irreducible; they cannot have a common
component. Therefore, $f_\Im(x,y)$ must be a non-zero real multiple
of $f_\Im(x,y)$.
Hence, we can write $f_\Im(x,y) = \lambda \cdot f_\Re(x,y)$ for
some $\lambda \in \mathbb{R}$. Therefore, $f(x,y) = (1 + \mathrm{i}
\lambda) f_\Re(x,y)$, i.e.\ $f(x,y)$ has only real coefficients up to
multiplication by a unit in $\mathbb{C}$, which we can cancel from the
equation $f(x,y) = 0$.
\end{proof}

\begin{lemma}
\label{lem:orientation}
Consider a plane algebraic curve $\mathcal{K}\colon p (x, y) = 0$, where
$p (x, y)$ is a polynomial with only real-valued coefficients. The paths
of the $y$-values on the complex plane algebraic curve $\mathcal{K}$
under complex conjugate movement of $x$ are complex conjugate.
\end{lemma}

\begin{proof}
$\mathcal{K}$ is the zero set of a
polynomial $p(x,y)$ with only real coefficients. Real coefficients are
invariant under complex conjugation. Hence, if we consider the complex
conjugate of the defining equation of $\mathcal{K}$, we find that \[0 =
\overline{0} = \overline{p(x,y)} = p(\overline{x}, \overline{y}),\] and
$p(\overline{x}, \overline{y})$ vanishes if and only if $p(x,y)$ vanishes.

Let $x(t)$ be a parameterization of the movement of $x$ through the
complex plane. Let $y(t)$ be a parameterization of the corresponding
analytic motion of a $y$-value so that $p(x(t), y(t)) = 0$, for all $t$.
Then to complex conjugate movement of $x$, parameterized by
$\overline{x(t)}$, corresponds complex conjugate motion of the
$y$-value, parameterized by $\overline{y(t)}$, since \[p(\overline{x(t)},
\overline{y(t)}) = 0 \Leftrightarrow p(x(t), y(t)) = 0.\]
(The $y$-values must be the same in both cases since they must agree
at all real points of the complex plane algebraic curve $\mathcal{K}$,
where complex conjugation has no effect.) \qedhere
\end{proof}

\begin{remark}
\label{rem:separate-or-together}
Suppose we can construct the tracer from the mover (and possibly further
elements) using only geometric operations that have an algebraic
representation. Moreover, suppose that the motion of the mover is
rationally parameterized in time parameter $t$. Then the $x$-coordinate
and the $y$-coordinate of the tracer satisfy algebraic equations $g(t,
x) = 0$ and $h(t, y) = 0$.
In practice, it may often be too expensive to work these equations out
symbolically. But in principle, we can determine an equation $f(x, y) = 0$
for the locus of the tracer by taking the $t$-resultant of $g(t, x)$
and $h(t, y)$ to eliminate $t$ from these equations.
\end{remark}

\begin{theorem}
\label{thm:orientation}
The locus generation algorithm is independent of the clockwise or
anticlockwise orientation of its complex detours.
\end{theorem}

\begin{proof}
By \autoref{rem:separate-or-together}, the $x$-coordinate of the tracer
is determined by an algebraic equation $g (t, x) = 0$. Analogously
to \autoref{ass:irreducible}, \autoref{ass:real-component}, and
\autoref{lem:real-coefficients}, we may assume that $g (t, x)$ has only
real coefficients, without loss of generality.
If we reverse the orientation of the complex detours, we obtain complex
conjugate movement of time parameter $t$.
By \autoref{lem:orientation}, the motion of $x$ such that $g (t, x) =
0$ under complex conjugate movement of $t$ is complex
conjugate. Particularly, the real values of $x$ resulting from either
movement agree.
The same argument applies to the algebraic equation $h (t, y) = 0$
that governs the motion of $y$ under movement of $t$.
Consequently, \autoref{alg:locus-generation} produces the same real
algebraic locus, independent of the clockwise or anticlockwise orientation
of its complex detours.
\end{proof}

\section{Termination}

Does the locus generation algorithm always terminate? Or can we get lost
on algebraic Riemann surfaces?

\noindent Analogously to the previous section, we assume that $f(x,y)$ is an
irreducible polynomial (cf.~\autoref{ass:irreducible}) with real coefficients
(cf.~\autoref{ass:real-component} and \autoref{lem:real-coefficients}).

Besides, recall the assumptions on which the locus generation algorithm
(\autoref{alg:locus-generation}) is based: Firstly, we assume that the complex
detours of \autoref{alg:locus-generation} are small enough so that they
wind around at most one ramification point of a coordinate of the tracer,
and only around ramification points at which the position of the tracer
is real-valued (\autoref{ass:at-most-one-singularity}).
Thus we preclude that the tracer jumps from one real arc of the real
algebraic locus to another due to a complex detour around a ramification
point of a coordinate at which the tracer has a complex position close
to the real algebraic locus.
Secondly, we assume that the steps we take on the complex detours are
small enough so that we can choose the right instance of the construction
by proximity (\autoref{ass:tracing-by-proximity}) and so that we do not
miss a real position of the tracer (\autoref{ass:no-miss}).

A proof of termination of \autoref{alg:locus-generation} has remained
elusive so far.
However, the successful application of Variant~A of
\autoref{alg:locus-generation} in Cinderella supports the following
conjecture:

\begin{conjecture}
\label{cnj:termination}
The locus generation algorithm terminates if it takes small enough
complex detours and small enough steps on every complex detour.
\end{conjecture}

\section{Generation of real connected components}

\begin{conjecture}
\label{cnj:connected-component}
Variant~B of \autoref{alg:locus-generation} generates an entire real
connected component of a real algebraic locus if it takes small
enough complex detours and small enough steps on every complex detour.
\end{conjecture}

\begin{remark}
Variant~A of \autoref{alg:locus-generation} need not generate an entire
real connected component of a real algebraic locus. It may miss real arcs
of a locus that correspond to non-real complex values of time parameter
$t$. For an example, see \autoref{sec:circle-line-projection}.
\end{remark}

\begin{remark}
A real connected component of a real algebraic locus need not be an
algebraic curve by itself. For example, consider a four-bar linkage with
bars of lengths $4$, $1$, $4$, and $2$. (We discuss how we can
express a mechanical linkage in terms of dynamic geometry in
\autoref{sec:watt-curves}.)
We leave the first bar of the linkage fixed and trace the midpoint of
the third bar under continuous movement of the linkage.
Then the four-bar linkage generates one of the two real connected
components of the plane algebraic sextic
\begin{align*}
\mathcal{C}\colon f(x,y) &= {(6 + 5 x - 2 x^3)}^2 + 3 (-45 + 4 x (-2 +
2 x + x^3)) y^2\\&\phantom{=}\qquad+ 4 (11 + 3 x^2) y^4 + 4 y^6 = 0.
\end{align*}
The algebraic curve $\mathcal{C}$ is irreducible (over the complex
numbers). Thus, each of the two real connected components by itself
cannot be an
algebraic curve.
\end{remark}

\section{Examples}

\subsection{Conic through five points in general position}

Consider Pascal's theorem (also known as `hexagrammum mysticum'):

\begin{theorem}[Pascal's theorem]
\label{thm:pascal}
If $A, B, C, D, E, F$ are six points on a conic, then opposite sides
of the hexagon $ABCDEF$ (extended to lines, if necessary) meet in three
collinear points.
\end{theorem}

\noindent For more details and proofs,
see~\cite{RichterGebert2011}*{esp.~Section~1.4 and Section~10.6}.

The converse of the theorem is also true, which gives rise to organic
generation of a conic through five points, by the following construction
(see~\autoref{fig:organic-conic}).

\begin{figure}[ht]
\centering
\begin{tikzpicture}[
    extended line/.style={shorten >=-#1,shorten <=-#1},
    extended line/.default={0.75cm}
]
\draw[name path=s,very thick,lightgray] (0,0) circle[x radius=3,y radius=2];
\coordinate (A) at ({3*cos(140)},{2*sin(140)});
\coordinate (B) at ({3*cos(-80)},{2*sin(-80)});
\coordinate (C) at ({3*cos(50)},{2*sin(50)});
\coordinate (D) at ({3*cos(-130)},{2*sin(-130)});
\coordinate (E) at ({3*cos(95)},{2*sin(95)});
\draw[extended line] (A) -- (B);
\draw[extended line] (D) -- (E);
\coordinate (F) at (intersection of A--B and D--E);
\coordinate (c1) at ($(F)+(3:4.5)$);
\coordinate (c2) at ($(F)+(183:2)$);
\path[name path=c1c2] (c1) -- (c2);
\path[name intersections={of=s and c1c2}];
\coordinate (cc1) at (intersection-1);
\coordinate (cc2) at (intersection-2);
\draw[extended line] (cc1) -- (cc2);
\draw[extended line] (B) -- (C);
\coordinate (G) at (intersection of c1--c2 and B--C);
\draw[extended line] (C) -- (D);
\coordinate (H) at (intersection of c1--c2 and C--D);
\coordinate (K) at (intersection of A--H and E--G);
\draw[extended line] (A) -- (K);
\draw[extended line] (E) -- (K);
\foreach \p in {A,B,C,D,E,F,G,H,K}
    \draw[draw=black,fill=white] (\p) circle (0.075);
\draw (A) node[below=2] {$A$};
\draw ($(B)+(-0.05,0.1)$) node[above=2] {$B$};
\draw (C) node[right=2] {$C$};
\draw (D) node[left=2] {$D$};
\draw (E) node[below=2] {$E$};
\draw (F) node[above=2] {$F$};
\draw ($(G)+(-0.05,0)$) node[above=2] {$G$};
\draw (H) node[above=2] {$H$};
\draw ($(K)+(0.05,0)$) node[above=2] {$K$};
\draw ($(A)!1.22!(B)$) node {$a$};
\draw ($(D)!1.25!(E)$) node {$b$};
\draw ($(cc2)+(183:0.95)$) node {$c$};
\draw ($(B)!1.25!(C)$) node {$d$};
\draw ($(C)!1.2!(D)$) node {$e$};
\draw ($(A)!1.17!(K)$) node {$f$};
\draw ($(K)!1.22!(E)$) node {$g$};
\end{tikzpicture}
\caption{An instance of \autoref{con:organic-conic}. When line~$c$
rotates about point~$F$, point~$K$ traces the conic through points $A,
B, C, D, E$.}
\label{fig:organic-conic}
\end{figure}
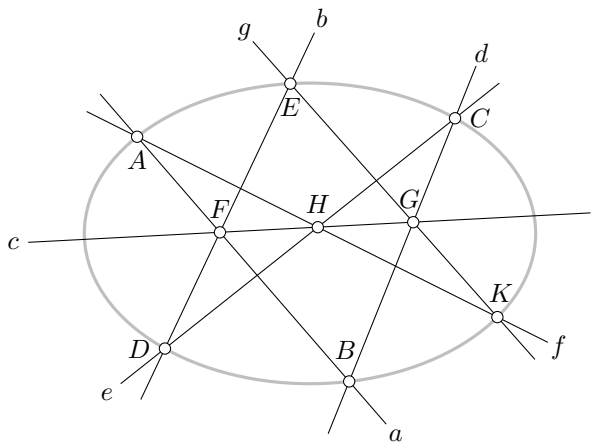

\begin{construction}
\label{con:organic-conic}
Let $A, B, C, D, E$ be five points of the real projective plane, in
general position.
\begin{enumerate}
\item Let $a$ be the line through $A$ and $B$.
\item Let $b$ be the line through $D$ and $E$.
\item Let $F$ be the intersection of $a$ and $b$.
\item Let $c$ be a line through $F$.
\item Let $d$ be the line through $B$ and $C$.
\item Let $G$ be the intersection of $c$ and $d$.
\item Let $e$ be the line through $C$ and $D$.
\item Let $H$ be the intersection of $c$ and $e$.
\item Let $f$ be the line through $A$ and $H$.
\item Let $g$ be the line through $E$ and $G$.
\item Let $K$ be the intersection of $f$ and $g$.
\end{enumerate}
When line $c$ rotates about point $F$, point $K$ traces the conic
through points $A, B, C, D, E$.
\end{construction}

\begin{remark}
\autoref{con:organic-conic} does not distinguish the orientation of line
$c$. Point $K$ returns to its initial position after half a turn of line
$c$ about point $F$. If line $c$ makes a full turn, point $K$ traces
the conic through points $A, B, C, D, E$ twice.
\end{remark}

\subsection{Orthogonal projection of a circle onto a line}
\label{sec:circle-line-projection}

We consider the following (seemingly simple) construction
(see~\autoref{fig:circle-line-projection}), because it
highlights the difference between Variant~A and Variant~B of
\autoref{alg:locus-generation}.

\begin{figure}[ht]
\centering
\begin{tikzpicture}[
    extended line/.style={shorten >=-#1,shorten <=-#1},
    extended line/.default={0.75cm}
]
\draw (0,0) circle (1.5);
\coordinate (A) at (30:1.5);
\coordinate (A1) at ($(30:1.5)+(-3,0)$);
\coordinate (A2) at ($(30:1.5)+(2,0)$);
\draw (A1) -- (A2);
\coordinate (B1) at (2.5,-2);
\coordinate (B2) at (2.5,2);
\draw (B1) -- (B2);
\draw[very thick,lightgray] (2.5,1.5) -- (2.5,-1.5);
\draw[dashed,help lines,thick] (0,1.5) -- (2.5,1.5);
\draw[dashed,help lines,thick] (0,-1.5) -- (2.5,-1.5);
\coordinate (B) at (intersection of A1--A2 and B1--B2);
\foreach \p in {A,B}
    \draw[draw=black,fill=white] (\p) circle (0.075);
\draw (A) node[above right] {$A$};
\draw (B) node[above right] {$B$};
\draw (A1) node[left] {$a$};
\draw (B2) node[above] {$b$};
\draw (-135:1.75) node {$c_0$};
\end{tikzpicture}
\caption{An instance of \autoref{con:circle-line-projection}.
When point $A$ moves around on circle $c_0$, point $B$ traces a segment
of line $b$.}
\label{fig:circle-line-projection}
\end{figure}
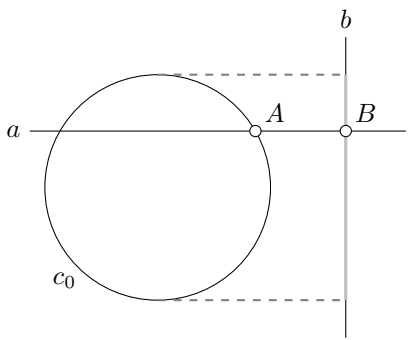

\clearpage

\begin{construction}
\label{con:circle-line-projection}
Let a line $b$, a circle $c_0$, and a point $A$ on circle $c_0$ be given.
\begin{enumerate}
\item Let $a$ be the line through $A$ perpendicular to $b$.
\item Let $B$ be the intersection of $a$ and $b$.
\end{enumerate}
When point $A$ moves around on circle $c_0$, point $B$ traces a segment
of line $b$.
\end{construction}

\begin{remark}
For simplicity, we use the geometric primitive `perpendicular to a line
through a point`. It can be easily constructed with compass and
straightedge (see~Book~I, Proposition~12 of Euclid's Elements).
\end{remark}

\begin{remark}
Line $a$ intersects circle $c_0$ in two points (counted with
multiplicity). Hence there are two positions of point $A$ (counted with
multiplicity) for every position of point $B$ on the segment that point
$B$ traces on line $b$. In other words, point $B$ covers the segment
twice as point $A$ makes a full turn on circle~$c_0$.
\end{remark}

\noindent Let us work out the real algebraic locus of point $B$ under
movement of point $A$ on circle $c_0$ algebraically. Without loss of
generality, let $c_0$ be the unit circle and $b$ the line parallel to
the $y$-axis intersecting the $x$-axis at $x = 2$. Let~$O$ denote the origin.
We need to parameterize the motion of point $A$ on circle~$c_0$. To that
end, we can use trigonometric functions, as follows:
\[A = {(\cos \varphi, \sin \varphi)}^\top, \quad -\pi \leq \varphi \leq \pi.\]
However, this parameterization is not rational. We use tangent half-angle
substitution,
\[t = \tan \frac{\varphi}{2}, \quad \cos \varphi = \frac{1 - t^2}{1 +
t^2}, \quad \sin \varphi = \frac{2t}{1 + t^2},\]
and homogeneous coordinates to derive the rational parameterization
\[A = {(1 - t^2, 2t, 1 + t^2)}^\top, \quad t \in \mathbb{R}.\]
Line $b$ has homogeneous coordinates \[b = {(1, 0, -2)}^\top.\]
Line $a$ is perpendicular to line $b$ and therefore has homogeneous
coordinates of the form \[a = {(0, 1, z(t))}^\top,\] where $z(t)$ has to
be determined so that point $A$ lies on line $a$, i.e.\ so that
\[\langle a, A\rangle = 2t + z(t) (1 + t^2) = 0.\]
Hence, homogeneous coordinates of line $a$ are \[a = {(0, 1 + t^2,
-2t)}^\top.\]
We obtain homogeneous coordinates of point $B$ by taking the cross
product of line $a$ and line $b$,
\[B = a \times b = {(2 (1 + t^2), 2t, 1 + t^2)}^\top \sim {\left(2,
\frac{2t}{1 + t^2}, 1\right)}^\top.\]
The real algebraic locus of point $B$ under movement of point $A$ on
circle $c_0$ is described by the implicit equations
\begin{gather*}
x - 2 = 0\\
t^2 y - 2 t + y = 0.
\end{gather*}
If we take the $t$-resultant of these equations, we arrive at a single
equation for the real algebraic locus,
\[{(x - 2)}^2 = 0.\]

\noindent The construction has singularities at $t = \pm 1$, where the
$y$-coordinate of point $B$ equals $\pm 1$. If we solve the equation
between $t$ and the $y$-coordinate of point $B$, \[t^2 y - 2t + y = 0,\]
for $t$, we find that
\[t = \frac{1 \pm \sqrt{1 - y^2}}{y}.\]
For real $y$-values of absolute values greater than $1$, $t$ becomes
complex. Point $B$ moves higher or lower than the (real-valued) extreme
positions of point $A$ on circle $c_0$ if and only if we allow point $A$
to become complex.

Variant~A of \autoref{alg:locus-generation} does not allow this to
happen and skips the branch of the real algebraic locus where point $B$ has
real-valued coordinates, but point $A$ (and thus $t$) is complex-valued.
It generates that part of the real algebraic locus where both mover and
tracer have real-valued coordinates.

Variant~B of \autoref{alg:locus-generation} does not skip the branch of
the real algebraic locus where point $B$ has real-valued coordinates,
but point $A$ (and thus $t$) is complex-valued. It generates the entire
real algebraic locus.

Variant~A seems more appropriate from the perspective of real projective
geometry; Variant~B seems more appropriate from the algebraic perspective.

\subsection{Conchoid of Nicomedes}

The conchoids of Nicomedes are a family of quartic plane algebraic curves
\[\mathcal{C}\colon f(x,y) = {(y + a)}^2 (x^2 + y^2) - b^2 y^2 = 0, \quad a,
b > 0.\]
For organic generation of a conchoid, we can use the following construction.

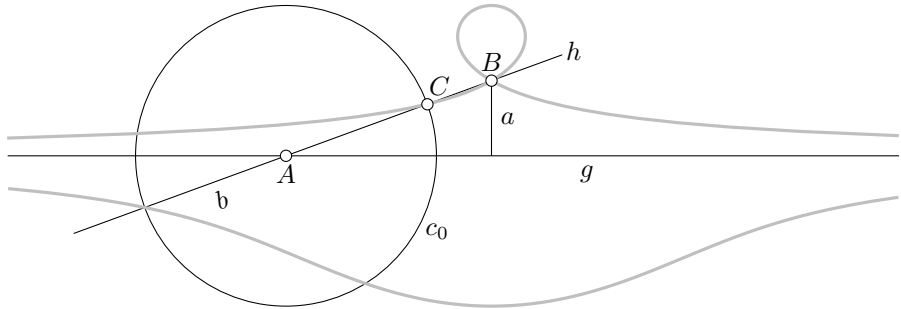
\begin{figure}[ht]
\centering
\begin{tikzpicture}[
    extended line/.style={shorten >=-#1,shorten <=-#1},
    extended line/.default={0.75cm}
]
\def\r{2}
\def\alpha{30}
\coordinate (A) at (-1,0);
\coordinate (B) at (\alpha:\r);
\draw (-4.7,0) -- (7.15,0);
\draw (B) -- ({\r*cos(\alpha)},0) node[midway,right] {$a$};
\draw[name path=c] (A) circle (\r);
\draw[name path=g,shorten >=-3cm,shorten <=-1cm] (B) -- (A);
\path[name intersections={of=c and g}];
\coordinate (C) at (intersection-1);
\coordinate (D) at ($(A)!-1!(C)$);
\path (A) -- (D) node[midway,sloped,below] {$b$};
\draw[very thick,domain=-77.4:74,samples=100,lightgray] plot
({\r*cos(\alpha)+(\r*sin(\alpha)/cos(\x)+\r)*sin(\x)},
{\r*sin(\alpha)-(\r*sin(\alpha)/cos(\x)+\r)*cos(\x)});
\draw[very thick,domain=-83.23:82.3,samples=100,lightgray] plot
({\r*cos(\alpha)+(\r*sin(\alpha)/cos(\x)-\r)*sin(\x)},
{\r*sin(\alpha)-(\r*sin(\alpha)/cos(\x)-\r)*cos(\x)});
\foreach \p in {A,B,C}
    \draw[draw=black,fill=white] (\p) circle (0.075);
\draw (3,0) node[below] {$g$};
\draw ($(A)!1.4!(B)$) node {$h$};
\draw (A) node[below] {$A$};
\draw (B) node[above] {$B$};
\draw ($(C)+(-0.1,0)$) node[above right] {$C$};
\draw ($(A)+(-30:\r)$) node[right] {$c_0$};
\end{tikzpicture}
\caption{An instance of \autoref{con:conchoid}.
When point $A$ moves along line~$g$, point $C$ traces the conchoid
with \emph{pole}~$B$, \emph{base}~$g$ and \emph{distance}~$b$.}
\label{fig:conchoid}
\end{figure}

\begin{construction}[conchoid of Nicomedes]
\label{con:conchoid}
Let $A$ be a point on a line $g$. Let $B$ be a point at distance $a >
0$ from $g$. Let $c_0$ be a circle of radius $b$ centred at $A$.
\begin{enumerate}
\item Let $h$ be the line through $A$ and $B$.
\item Let $C$ be an intersection of $c_0$ and $h$.
\end{enumerate}
When point $A$ moves along line $g$, point $C$ traces the conchoid
with \emph{pole}~$B$, \emph{base}~$g$, and \emph{distance}~$b$.
\end{construction}

\noindent As mentioned in the introduction, part of the original
motivation to study the conchoid of Nicomedes was that it can be used
for angle trisection. The following construction is
based on~\cite{EFMR}*{Conchoïde de Nicomède,
\url{http://www.mathcurve.com/courbes2d/conchoiddenicomede/conchoiddenicomede.shtml}}.

\begin{figure}[ht]
\centering
\begin{tikzpicture}[
    extended line/.style={shorten >=-#1,shorten <=-#1},
    extended line/.default={0.75cm}
]
\def\r{2}
\def\alpha{30}
\coordinate (E) at (0,0);
\coordinate (D) at (\alpha:\r);
\coordinate (F) at ({\r+1},0);
\coordinate (G) at ({180-\alpha/3}:\r);
\coordinate (H) at (intersection of E--F and D--G);
\draw (D) -- (E);
\draw (E) -- (G);
\draw[extended line] (D) -- (H);
\draw[extended line] (H) -- (6.4,0);
\draw (E) circle (\r);
\draw[very thick,domain=-77.4:74,samples=100,lightgray] plot
({\r*cos(\alpha)+(\r*sin(\alpha)/cos(\x)+\r)*sin(\x)},
{\r*sin(\alpha)-(\r*sin(\alpha)/cos(\x)+\r)*cos(\x)});
\draw[very thick,domain=-83.23:82.3,samples=100,lightgray] plot
({\r*cos(\alpha)+(\r*sin(\alpha)/cos(\x)-\r)*sin(\x)},
{\r*sin(\alpha)-(\r*sin(\alpha)/cos(\x)-\r)*cos(\x)});
\foreach \p in {D,E,F,G,H}
    \draw[draw=black,fill=white] (\p) circle (0.075);
\draw ($(D)+(0.075,0)$) node[above] {$D$};
\draw (E) node[below] {$E$};
\draw (F) node[below right] {$F$};
\draw (G) node[above left] {$G$};
\draw (H) node[below] {$H$};
\draw ($(E)+(-30:\r)$) node[right] {$c_1$};
\draw ($(E)!1.75!(F)$) node[below] {$l$};
\draw ($(G)!1.27!(D)$) node {$m$};
\end{tikzpicture}
\caption{An instance of \autoref{con:trisection}. By construction,
angle $\angle HEG$ trisects angle $\angle DEF$.}
\label{fig:trisection}
\end{figure}
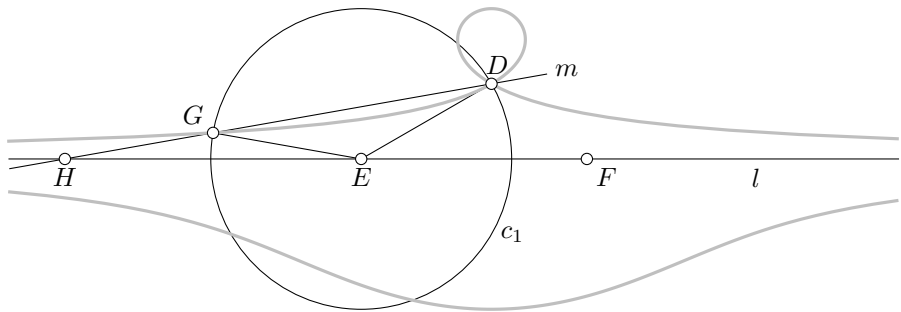

\begin{construction}[angle trisection]
\label{con:trisection}
Let $\angle DEF$ be the angle to be trisected. We can trisect a right
angle by constructing an equilateral triangle. Hence, without loss of
generality, let angle $\angle DEF$ be acute.
\begin{enumerate}
\item Let $l$ be the line through $E$ and $F$.
\item Use \autoref{con:conchoid} to generate the conchoid with pole $D$,
base $l$, and distance $|DE|$.
\item Let $c_1$ be the circle centred at $E$, through $F$. 
\item Let $G$ be the intersection of $c_1$ with the conchoid that
lies on the same side of $l$ as $F$.
\item Let $m$ be the line through $D$ and $G$.
\item Let $H$ be the intersection of $l$ and $m$.
\end{enumerate}
Then angle $\angle HEG$ trisects angle $\angle DEF$.
\end{construction}

\begin{proof}
By construction of the conchoid, points $H$ and $E$ are equidistant to
point~$G$. Therefore, triangle $\triangle EGH$ is equilateral and \[\angle
HEG = \angle GHE.\] We apply the exterior angle theorem to triangle
$\triangle EGH$ and find that \[\angle DGE = \angle GHE + \angle HEG =
2 \cdot \angle HEG.\]
Triangle $\triangle GED$ is equilateral by construction, and thus
\[\angle EDG = \angle DGE = 2 \cdot \angle HEG.\]
We apply the exterior angle
theorem to triangle $\triangle HED$ and conclude
\begin{align*}
\angle DEF &= \angle GHE + \angle EDH = \angle GHE + \angle EDG = \angle
HEG + 2 \cdot \angle HEG\\
&= 3 \cdot \angle HEG. \qedhere
\end{align*}
\end{proof}

\subsection{Watt curves}
\label{sec:watt-curves}

We can use \autoref{alg:locus-generation} to simulate
mechanical linkages. For example, the following construction uses a
four-bar linkage to generate a Watt curve, a plane algebraic sextic
\[\mathcal{C}\colon f(x,y) = (x^2 + y^2) {(x^2 + y^2 - a^2 - b^2 +
c^2)}^2 + 4 a^2 y^2 (x^2 + y^2 - b^2) = 0\] with parameters $a, b, c > 0$.

\begin{figure}[ht]
\centering
\begin{tikzpicture}
\def\a{2}
\def\b{2.5}
\def\c{1.5}
\coordinate (A) at (-\a,0);
\coordinate (B) at (\a,0);
\coordinate (C) at ($(A)+(100:\b)$);
\draw[help lines,black] (A) circle (\b);
\draw[name path=cB,help lines,black] (B) circle (\b);
\draw[name path=cC,help lines,black] (C) circle ({2*\c});
\path[name intersections={of=cC and cB}];
\coordinate (D) at (intersection-1);
\coordinate (E) at ($(C)!0.5!(D)$);
\draw[thick] (A) -- (C) -- (D) -- (B) -- cycle;
\draw[very thick,domain=-138.5945:-41.4045,samples=150,lightgray] plot
(xy polar cs:angle=\x,
radius={sqrt(\b^2-(sqrt(\c^2-\a^2*cos(\x)^2)+\a*sin(\x))^2)});
\draw[very thick,domain=41.4055:53.1305,samples=150,lightgray] plot
(xy polar cs:angle=\x,
radius={sqrt(\b^2-(sqrt(\c^2-\a^2*cos(\x)^2)+\a*sin(\x))^2)});
\draw[very thick,domain=126.87:138.595,samples=150,lightgray] plot
(xy polar cs:angle=\x,
radius={sqrt(\b^2-(sqrt(\c^2-\a^2*cos(\x)^2)+\a*sin(\x))^2)});
\draw[very thick,domain=-138.5945:-41.4045,samples=150,lightgray] plot
(xy polar cs:angle=\x,
radius={-sqrt(\b^2-(sqrt(\c^2-\a^2*cos(\x)^2)+\a*sin(\x))^2)});
\draw[very thick,domain=41.4055:53.1305,samples=150,lightgray] plot
(xy polar cs:angle=\x,
radius={-sqrt(\b^2-(sqrt(\c^2-\a^2*cos(\x)^2)+\a*sin(\x))^2)});
\draw[very thick,domain=126.87:138.595,samples=150,lightgray] plot
(xy polar cs:angle=\x,
radius={-sqrt(\b^2-(sqrt(\c^2-\a^2*cos(\x)^2)+\a*sin(\x))^2)});
\foreach \p in {A,B,C,D,E}
    \draw[draw=black,fill=white] (\p) circle (0.075);
\draw (A) node[below] {$A$};
\draw (B) node[below] {$B$};
\draw (C) node[above] {$C$};
\draw (D) node[pin=75:$D$] {};
\draw (E) node[pin={$E$}] {};
\draw ($(A)+(135:\b)$) node[left] {$c_0$};
\draw ($(B)+(45:\b)$) node[right] {$c_1$};
\draw ($(C)+(45:{2*\c})$) node[right] {$c_2$};
\end{tikzpicture}
\caption{An instance of \autoref{con:watt-curve}.
When point $C$ moves on circle~$c_0$ according to
\autoref{alg:locus-generation}, point $E$ traces a Watt curve with
parameters $a = |AB|/2$, $b = |AC| = |BD|$, and $c = |CD|/2$.}
\label{fig:watt-curve}
\end{figure}
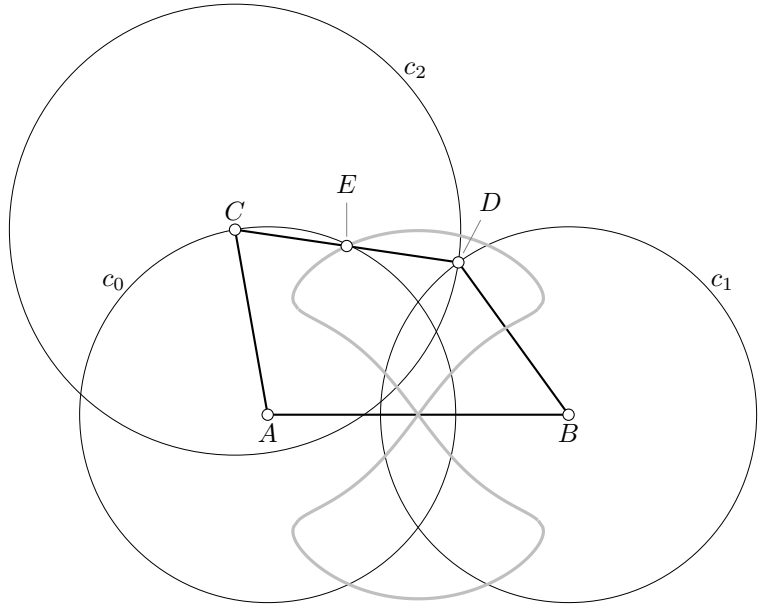

\begin{construction}
\label{con:watt-curve}
Let $A$ and $B$ be two points in the plane at distance $2 a$ from each
other. Let $c_0$ be a circle with centre $A$ and radius $b$. Let $c_1$
be a circle with centre $B$ and radius $b$.
\begin{enumerate}
\item Let $C$ be a point on $c_0$.
\item Let $c_2$ be a circle with centre $C$ and radius $2 c$.
\item Let $D$ be an intersection of $c_1$ and $c_2$.
\item Let $E$ be the midpoint of the segment between $C$ and $D$.
\end{enumerate}
When point $C$ moves on circle $c_0$ according to
\autoref{alg:locus-generation}, point $E$ traces a Watt curve with
parameters $a, b, c$.
\end{construction}

\begin{remark}
The first bar of the four-bar linkage, segment $AB$, has length $2
a$. Circles $c_0$, $c_1$, and $c_2$ have fixed radii. Hence, they
prescribe the lengths $|AC| = b$, $|CD| = 2c$, and $|DB| = b$ of the
remaining bars of the four-bar linkage.
\end{remark}

\begin{remark}
Depending on parameters $a, b, c$, Watt curves have a wide variety of
different shapes.
\end{remark}

\begin{remark}
In \autoref{fig:watt-curve}, we choose parameters $a, b, c$ so that
$a > c$. Therefore, the movement of point $C$ on circle $c_0$ is
constrained. If point $C$ moved too far to the left, then circles $c_1$
and $c_2$ would move apart and would no longer have real-valued
intersections. Such movement is not possible without breaking the linkage.
\autoref{alg:locus-generation} resolves the singularities when circles
$c_1$ and $c_2$ merely touch each other, by taking complex detours
around them. At every such singularity, point $C$ reverses its direction
of movement on circle $c_0$.
Apparently, \autoref{alg:locus-generation} produces a physically
reasonable motion of the four-bar linkage.
\end{remark}

\section*{Funding}

This research was supported by DFG Collaborative Research Center TRR 109,
``Discretization in Geometry and Dynamics''.


\begin{bibdiv}
\begin{biblist}

\bib{BrieskornKnoerrer1986}{book}{
    title={Plane algebraic curves},
    author={Brieskorn, Egbert},
    author={Knörrer, Horst},
    translator={Stillwell, John},
    language={English},
    date={1986},
    publisher={Birkhäuser},
    address={Basel}
}

\bib{EFMR}{misc}{
    title={Encyclopédie des formes mathématiques remarquables},
    author={Ferréol, Robert},
    author={Mandonnet, Jaques},
    url={http://www.mathcurve.com},
    note={\url{http://www.mathcurve.com}},
    date={2005}
}

\bib{Kortenkamp1999}{thesis}{
	title={Foundations of Dynamic Geometry},
	author={Kortenkamp, Ulrich},
	type={dissertation},
	organization={ETH Zürich},
	date={1999},
	address={Zurich}
}

\bib{KortenkampRichterGebert2001b}{article}{
	title={Grundlagen dynamischer Geometrie},
    author={Kortenkamp, Ulrich},
    author={Richter-Gebert, Jürgen},
	book={
		title={Zeichnung -- Figur -- Zugfigur},
		subtitle={Mathematische und didaktische Aspekte
		dynamischer Geometrie-Software},
		editor={Elschenbroich, H.-J.},
		editor={Gawlick, Th.},
		editor={Henn, H.-W.},
		publisher={Franzbecker},
		address={Hildesheim}
	},
	pages={123--144},
	date={2001}
}

\bib{KortenkampRichterGebert2002}{article}{
    title={Complexity issues in dynamic geometry},
    author={Kortenkamp, Ulrich},
    author={Richter-Gebert, Jürgen},
    book={
        title={Festschrift in the honor of Stephen Smale's 70th birthday},
        editor={Rojas, M.},
        editor={Cucker, Felipe},
        publisher={World Scientific}
    },
    pages={355--404},
    date={2002}
}

\bib{KortenkampRichterGebert2006}{misc}{
	title={Cinderella},
	subtitle={The interactive geometry software},
	author={Kortenkamp, Ulrich},
	author={Richter-Gebert, Jürgen},
	date={2006},
	note={\url{http://www.cinderella.de}}
}

\bib{Lindemann1882}{article}{
    title={Ueber die Zahl $\pi$},
    author={Lindemann, Ferdinand},
    journal={Mathematische Annalen},
    volume={20},
    number={2},
    pages={213--225},
    date={1882},
    doi={10.1007/BF01446522}
}

\bib{Loria1910}{book}{
    author={Loria, Gino},
    translator={Schütte, Fritz},
    language={German},
    title={Spezielle algebraische und transzendente ebene Kurven},
    subtitle={Theorie und Geschichte},
    volume={1},
    publisher={Teubner},
    address={Leipzig},
    edition={2},
    date={1910}
}

\bib{RichterGebert2011}{book}{
	title={Perspectives on Projective Geometry},
	subtitle={A Guided Tour Through Real and Complex Geometry},
	author={Richter-Gebert, Jürgen},
	date={2011},
	publisher={Springer},
	address={Berlin},
    doi={10.1007/978-3-642-17286-1}
}

\bib{RichterGebert2014}{misc}{
    author={Richter-Gebert, Jürgen},
    note={Personal communication},
    date={2014}
}

\bib{Wantzel1837}{article}{
    title={Recherches sur les moyens de reconnaître si un problème de
    Géométrie peut se résoudre avec la règle et le compas},
    author={Wantzel, Pierre Laurent},
    journal={Journal de Mathématiques Pures et Appliquées},
    volume={1},
    number={2},
    pages={366--372},
    date={1837}
}

\end{biblist}
\end{bibdiv}
\end{document}